\documentclass[11pt,leqno]{amsart}

\usepackage{amsfonts,amssymb}
\usepackage{graphicx,tikz}
\usepackage{tikz-cd} 
\usepackage{float}
\usepackage{amsmath, amsthm, amsfonts}
\usepackage{hyperref}
\usepackage{enumitem}
\usepackage[capitalise]{cleveref}
\usepackage{comment}
\usepackage{enumitem}
\usepackage{amsrefs}
\usepackage{nicefrac}
\usepackage[official]{eurosym}
\usepackage{graphicx}
\usepackage{nomencl}
\usepackage{amsfonts}
\usepackage{hyperref}
\usepackage{tabto}
\usepackage{amssymb}
\usepackage{fancyhdr}
\usepackage{amscd}
\usepackage[english]{babel}

\usepackage[margin=1.15in]{geometry}

\theoremstyle{plain}

\newtheorem*{theorem*}{Theorem}
\newtheorem{theorem}{Theorem}[section]
\newtheorem{proposition}[theorem]{Proposition}
\newtheorem{corollary}[theorem]{Corollary}
\newtheorem{lemma}[theorem]{Lemma}

\theoremstyle{definition}

\newtheorem{question}{\scshape{Question}}

\newtheorem{remark}[theorem]{\scshape{Remark}}

\newtheorem*{definition*}{\scshape{Definition}}

\newcommand\Z{{\mathbb{Z}}}

\makeatletter
\def\cleardoublepage{\clearpage\if@twoside \ifodd\c@page\else
	\hbox{}
	\thispagestyle{empty}
	\newpage
	\if@twocolumn\hbox{}\newpage\fi\fi\fi}
\makeatother
\DeclareMathOperator{\Aut}{Aut}
\DeclareMathOperator{\St}{St}

\DeclareMathOperator{\st}{St}

\numberwithin{equation}{section}

\author[C.\,G. Cox]{Charles Garnet Cox} 
\address{Charles Garnet Cox: 
School of Mathematics, University of Bristol, Bristol BS8 1UG, United Kingdom}
\email{charles.cox@bristol.ac.uk}

\author[A. Thillaisundaram]{Anitha Thillaisundaram} 
\address{Anitha Thillaisundaram: 
Centre for Mathematical Sciences, Lund University,  223 62 Lund, Sweden}
\email{anitha.thillaisundaram@math.lu.se}

 

\begin{document}
	
	\title{Invariable generation of certain branch groups}

\begin{abstract}
Let $G$ be a group. Then $S\subseteq G$ is an invariable generating set of $G$ if every subset $S'$ obtained from $S$ by replacing each element with a conjugate is also a generating set of $G$.
We investigate invariable generation among key examples of branch groups. In particular, we prove that all generating sets of the torsion Grigorchuk groups, of the branch Grigorchuk-Gupta-Sidki groups and of the torsion multi-EGS 
 groups (which are natural generalisations of the Grigorchuk-Gupta-Sidki 
 groups) are invariable generating sets. Furthermore, for the first  Grigorchuk group and the torsion Grigorchuk-Gupta-Sidki groups, every finitely generated subgroup has a finite invariable generating set. Our results apply to finitely generated groups in $\mathcal{MN}$, the class of groups whose maximal subgroups are all normal. We then obtain that any $2$-generated group in $\mathcal{MN}$ is almost $\frac32$-generated, and end by applying this observation to generating graphs.
\end{abstract}

	\maketitle

\section{Introduction}
As introduced in 1988 by Dixon \cite{Dixon}, a group $G$ is \emph{invariably generated} (IG) if there exists a set $S\subseteq G$ such that for any choice of elements $a_s \in G$,  where $s \in S$, we have $\langle a_s^{-1}sa_s\;:\;s \in S\rangle=G$. We then say that $S$ is an \emph{IG-set}. This notion is significantly older, with Jordan in 1872 showing that the following, equivalent, property holds for every finite group $G$: for each transitive action of $G$ on a set $X$ where $|X|>1$, there exists $g\in G$ which is fixed-point-free on $X$. For Galois groups, where elements are considered up to conjugacy, the notion of IG-set is natural. As a generation property, it has several known unexpected properties. An introduction to invariable generation for infinite groups can be found in \cite{KLS}. We summarise some key results below. 

First, from \cite{KLS11}, every finite group $G$ has an IG-set of size at most $\log_2|G|$. Also, there are finitely generated groups which are not IG, e.g. non-abelian free groups of finite rank (see \cite{Wiegold1}). One can also distinguish between groups for which a finite IG-set exists (making them \emph{finitely invariably generated}, or FIG for short). Independently, \cites{ashot, exam2} showed that there are finitely generated groups that are IG but not FIG, a question posed in~\cite{Cox5}. Neither FIG nor IG are preserved across subgroups, shown strongly within \cite{Cox5} by  the fact that any finitely generated group $G$ results in $G\wr \Z$ being FIG.  From \cite{Wiegold1} the property of being IG is closed under extensions, meaning every virtually solvable group is IG (although, as asked in \cite{KLS}, it is unknown whether all finitely generated solvable groups are FIG). Both Wiegold~\cite{Wiegold1} and the authors in~\cite{KLS} asked whether IG is also stable under finite-index subgroups. Surprisingly, the answer is no as shown, again independently (for an index 2 subgroup), in both \cites{ashot} and \cite{exam2}. Below are some currently open questions relating to invariable generation.

\begin{question}  Is every finitely generated amenable group IG?
\end{question}
\begin{question} Under which conditions does invariable generation pass to finite-index subgroups?
\end{question}
\begin{question} Can one find more natural examples of finitely generated IG but not FIG groups?
\end{question}
Given our investigations, we also pose the following new question.
\begin{question}  Is every branch group (and even weakly branch group) IG? 
\end{question}

For $G$ a group, let $d(G)$ (respectively $d_I(G)$ if $G$ is FIG) denote the minimal cardinality of a set of (invariable) generators for $G$. It was shown in~\cite[Prop.~2.5]{KLS11} that $d_I(G)-d(G)$ can be arbitrarily large for finite groups. In~\cite[Thm.~A]{Cox6} it was shown that if $G$ is a finitely generated Houghton group, then $d(G)=d_I(G)$. 
\begin{question} Under which conditions do we have $d_I(G)=d(G)$ for a group $G$?
\end{question}

It is hoped that the class of branch and weakly branch groups may be suitably wild to provide examples for the questions above, since this class contains groups of intermediate word growth, infinite torsion groups, amenable but not elementary amenable groups, and just infinite groups. These classes of groups appeared from constructions of groups by Grigorchuk and by Gupta and Sidki in the 1980s, which then led to a generalised family of so-called Grigorchuk-Gupta-Sidki groups (GGS-groups for short), and also to the family of multi-EGS groups (where EGS stands for Extended Gupta Sidki).

One of the major aims in the subject of branch groups is to classify finitely generated branch groups according to the existence of a maximal subgroup of infinite index. Such a characterisation has been achieved for other families of groups, e.g., the classical theorem of Margulis and Soifer, which states that a finitely generated linear group has maximal subgroups only of finite index if and only if it is virtually solvable. The first results for branch groups in this direction were  by Pervova~\cites{Pervova3, Pervova4}, who proved that the  torsion Grigorchuk groups and the torsion GGS-groups do not contain maximal subgroups of infinite index. Pervova's result was generalised to the torsion multi-EGS groups by Klopsch and Thillaisundaram~\cite{KT}. Francoeur and Thillaisundaram~\cite{FT} have further shown that the non-torsion GGS-groups do not have maximal subgroups of infinite index. The first explicit examples of finitely generated branch groups with maximal subgroups of infinite index were  provided by Francoeur and Garrido~\cite{FG}. Their examples are the non-torsion \v{S}uni\'{c} groups  acting on the binary rooted tree. 
 
In this note we use known results for maximal subgroups of branch groups to obtain results for invariable generation. In particular, it was proven in \cite[Lem.~2.6]{KLS} that for  a finitely generated group $G$, if all maximal subgroups of $G$ have finite index, then $G$ is IG. Furthermore,   if there exists an integer $c$ such that every maximal subgroup $M$ of $G$ satisfies $|G:M|\le c$, then $G$ is FIG. This means that knowledge about maximal subgroups can give some information about the properties IG and FIG. (Observe, however, that the previous results are not necessary conditions. Obviously $\Z$ is FIG rather than just IG, but has maximal subgroups of unbounded index. Also, from~\cite{Cox6}, the Houghton groups $H_2$, $H_3, \ldots$ are all FIG, but all have maximal subgroups of infinite index; here, the stabiliser of a single point $x$ gives such a maximal subgroup.)  This note considers groups that are one of the following:
\begin{itemize}
\item[(i)] the first Grigorchuk group;
\item[(ii)] the torsion Grigorchuk groups;
\item[(iii)] torsion GGS-groups
\item[(iv)] branch GGS-groups;
\item[(v)] torsion multi-EGS groups; and
\item[(vi)] non-torsion \v{S}uni\'{c} groups acting on the binary rooted tree.
\end{itemize}
Note that the first Grigorchuk group is also in the family (ii) above. It was mentioned in the introduction of \cite{KLS} that the standard generating set for the first Grigorchuk group  is also an IG-set. A stronger result holds, as seen in (1) below. Result (1) uses the work in \cites{Pervova3, Pervova4, FT, KT}.
\begin{theorem}\label{thm:mainlist}
If $G$ is in
\begin{itemize}
\item[\textup{(1)}]\label{cor:any-gen-set} families (i)--(v), then any generating set of $G$ is an IG-set.
\item[\textup{(2)}]\label{thm:fg-subgroups} families (i) and (iii), then every finitely generated subgroup of $G$ is FIG.
\item[\textup{(3)}] family (v) (which includes family (iii)), then every finite-index subgroup of $G$ is FIG.
\item[\textup{(4)}]\label{thm:Sunic} family (vi), then $G$ is FIG and has finite generating sets that are not IG-sets. 
\end{itemize}
\end{theorem}
\begin{remark} We make some comments on the above results:
\begin{itemize}
\item As far as the authors are aware, the behaviour of (1) above has only been documented for  finite nilpotent groups; see \cite[Prop.~2.4]{KLS11}.
\item For (2) we use the subgroup structure (specifically the subgroup induction property) of the first Grigorchuk group and of a torsion GGS-group. This mirrors the situation for nilpotent groups; cf. Proposition~\ref{prop:anynilp}. 
\item That (3) provides an answer to Question 2 for  torsion multi-EGS groups, which are all branch. Previously~\cite{Cox6} had shown that all finite-index subgroups of the Houghton groups $H_n$, for $n\ge 2$, are FIG.
\item In (4) the final comment is immediate from these groups having maximal subgroups of infinite index, but that like the Houghton groups they are still FIG.
\item For any group $G$ in (i)-(vi), we get that $d(G)=d_I(G)$.
\end{itemize}
\end{remark} 
Next, a group $G$ is said to be \emph{$\frac32$-generated} if for every $g\in G\backslash\{1\}$ there is an $h\in G$ such that $\langle g, h\rangle=G$. This has been studied for over 50 years. We can also ask similar questions for general 2-generated groups by introducing the notion of a group~$G$ being \textit{almost} $\frac32$-generated, which means that if $g\in G$ is such that $gG'\in G/G'$ is part of a generating pair in $G/G'$, then $g$ is part of a generating pair in $G$. This idea was introduced in \cite{Golan}, where it was proved that Thompson's group~$F$ is almost $\frac32$-generated. It relates to an older, much stronger, condition: if $G$ is a group and $S\subseteq G$, then $S$ \emph{weakly generates} $G$ if the image of $S$ in $G/G'$ is a generating set for $G/G'$. Clearly if $S$ generates $G$, then $S$ weakly generates $G$. For the next result, recall that $\mathcal{MN}$ denotes the class of groups for which every maximal subgroup is normal. 
\begin{proposition}\label{prop:weaklygen} Let $G$ be such  that every proper subgroup lies in a  maximal subgroup. Then the conditions `every $S\subseteq G$ that weakly generates $G$ also generates $G$' and `$G\in\mathcal{MN}$' are equivalent.
\end{proposition}

\begin{corollary}\label{cor:3/2}
Let $G\in\mathcal{MN}$ with $d(G)=2$. Then $G$ is almost $\frac32$-generated.
\end{corollary}
\smallskip

\noindent\textit{Organisation}. In Section~\ref{sec:maximal} we make our key observation concerning when every generating set is an IG-set. We then apply this result to prove Theorem~\ref{thm:mainlist} in Section~\ref{sec:apps}.  Finally in Section~\ref{sec:3/2} we prove Proposition~\ref{prop:weaklygen} and apply this observation to understand the generating graph of 2-generated groups in $\mathcal{MN}$.

\medskip

\noindent\textbf{Acknowledgements.} We thank Dominik Francoeur, Benjamin Klopsch and Jeremy Rickard for very helpful conversations. We also thank the referees for suggesting improvements to the paper. This research was supported by a London Mathematical Society Research in Pairs (Scheme 4) grant.
 The first author would also like to thank the welcoming staff at Lund University during his two-week stay in Sweden. 

\section{Initial observations}\label{sec:maximal}
We begin with an elementary observation, which will provide strong properties for invariable generation for many branch groups.
\begin{proposition}\label{prop:normal}
Let $G$ be a  group where every proper subgroup lies in a maximal subgroup. Then $G\in\mathcal{MN}$ if and only if every generating set is an IG-set. 

Assume that $H$ is finitely generated. Then $H\in \mathcal{MN}$ if and only if every finite generating set of $H$ is an IG-set.
\end{proposition}
\begin{proof}
For the forward direction of the first statement, let $S$ be a generating set for $G$. If $\langle s_i^{\,g_i}\mid s_i\in S,\,g_i\in G\rangle\ne G$, then $\langle s_i^{\,g_i}\mid s_i\in S,\,g_i\in G\rangle\le M$ for some maximal subgroup $M$. As all maximal subgroups are normal, it follows that  $\langle S\rangle\le M$, a contradiction. 

For the reverse direction, suppose $M$ is a non-normal maximal subgroup. Then there exists an $x\in M$ such that $x^g\notin M$ for some $g\in G$. Hence $\langle M,x^g\rangle=G$, and so the set $\{x^g\}\cup M$ is a generating set but not an IG-set.

The forward direction of the final statement follows immediately from the first statement. We now assume that $H\not\in \mathcal{MN}$, and will show that $H$ has a finite generating set that is not an IG-set. Let $H$ have a non-normal maximal subgroup $M$. If $M$ is finitely generated, then the previous paragraph produces a finite generating set that is not an IG-set. So assume that $M$ is not finitely generated. Let $S$ be a generating set for~$M$. We have $\langle S\rangle =M$ and $\langle S\cup \{x\}\rangle =H$ for some $x\in H$ where $x^h\in M$ for some $h\in H$ (making $S\cup \{x\}$ a generating set but not an IG-set). However, $H$ is finitely generated and so any infinite generating set of $H$ contains a finite generating set. Hence $S\cup\{x\}$ contains a finite subset~$T$ where $\langle T\rangle=H$. Also $x\in T$ (since otherwise $\langle T\rangle\le\langle S\rangle= M$) and so $T$ is also not an IG-set.
\end{proof}

Note that the first hypothesis in this proposition is satisfied by any finitely generated group. Such a hypothesis is necessary; see Remark \ref{example:nonIGsets}. We note that Myropolska proved in \cite[Prop.~2.2]{Myropolska} that for a finitely generated group, all maximal subgroups are normal if and only if every normal generating set is a generating set. Hence, if $G$ is a finitely generated group with a non-normal maximal subgroup, then there exists some normal generating set that is not a generating set. The existence of such a normal generating set also then implies the existence of a generating set that is not an IG-set. We observe furthermore that since all maximal subgroups of infinite index are not normal, the existence of maximal subgroups of infinite index also indicates that there is a generating set that is not an IG-set. We see this behaviour for the non-torsion \v{S}uni\'{c} groups acting on a binary rooted tree, but this observation applies widely and provides a new viewpoint on finding subsets that are not IG-sets.

Recall that for a nilpotent group, all  maximal subgroups are normal; see \cite[Thm.~12.1.5]{Robinson}. Therefore, Proposition \ref{prop:normal} immediately shows that finitely generated nilpotent groups have every generating set an IG-set.  
This was recorded for finite nilpotent groups as part of a stronger result; see \cite[Prop.~2.4]{KLS11}.
The following, which was pointed out to us by B. Klopsch, shows that the finitely generated assumption can be dropped.
\begin{proposition}\label{prop:anynilp} Let $G$ be a nilpotent group. Then every generating set of $G$ is an IG-set.
\end{proposition}

\begin{proof} We proceed by induction on $c$, the nilpotency class of $G$. If $c=1$, then the result is clear. So suppose the result is true for groups of  class at most $c-1$, and assume that $G$ is of  class~$c$. Let $G$ be generated by a set $X=\{x_i\mid i\in I\}$, for some indexing set $I$. For each $i\in I$, let $g_i\in G$. By the induction hypothesis,  the set $\{x_i[x_i,g_i]\gamma_c(G)\mid i\in I
\}$ generates $G/\gamma_c(G)$. Equivalently, for every $g\in G\backslash \gamma_c(G)$, there is an element $h_g\in \langle x_i[x_i,g_i]\mid i\in I\rangle$ and $z_g\in \gamma_c(G)$ such that $h_g=gz_g$. Further, since $\gamma_c(G)$ is generated by the elements $[g,x]$, where $g\in \gamma_{c-1}(G)\backslash \gamma_c(G)$ and $x\in G\backslash \gamma_2(G)$, we can find $h_g,h_x\in \langle x_i[x_i,g_i]\mid i\in I\rangle$ such that $h_g=gz_g$ and $h_x=xz_x$, where $z_g,z_x\in \gamma_c(G)$. Then $[h_g,h_x]=[gz_g,xz_x]=[g,x]$, since  $\gamma_c(G)\le Z(G)$.
\end{proof}

\begin{remark}\label{example:nonIGsets} The above does not generalise to solvable groups. Take, for example, $C_2\ltimes(\mathbb{Q}\times \mathbb{Q})$ where the $C_2=\langle t\rangle$ permutes the two copies of $\mathbb{Q}$. This has one maximal subgroup, $M=\mathbb{Q}\times \mathbb{Q}$ which is normal, but not every generating set is an IG-set. Indeed, $\{t\}\cup\{(q, q') \in \mathbb{Q}^2\}$ is an IG-set whereas $\{t(q, q') : q, q'\in \mathbb{Q}\}$ has conjugates which all lie in $H=\langle t\rangle\ltimes\{(q, q) : q \in \mathbb{Q}\}$.
\end{remark}

Further, we record  the following additional tool for identifying IG groups, which is of independent interest. 

\begin{lemma}
Let $G$ be a group and suppose that $S\subseteq G$ is a generating set for $G$. Write $\widetilde{S}$ for the subset obtained from $S$ where each element  of $S$ is replaced by any conjugate of it. Then  $S$ is an IG-set if and only if $\langle \widetilde{S}\rangle\ge G'$ for all choices of $\widetilde{S}$.
\end{lemma}

\begin{proof}
The forward direction is clear, since if $\langle S\rangle=G$, then $\langle \widetilde{S}\rangle=G\ge G'$ for all choices of~$\widetilde{S}$. For the reverse direction, since   $\langle\widetilde{S}\rangle=\langle s^{g_s}\mid s\in S \rangle=\langle s[s,g_s]\mid s\in S \rangle$ for some fixed choice of $\{g_s\mid s \in S\}\subseteq G$, if $\langle\widetilde{S}\rangle\ge G'$, then it follows that  $\langle\widetilde{S}\rangle=\langle s\mid s\in S\rangle=G$. Since $\widetilde{S}$ was arbitrary, we have that $S$ is an IG-set, as required.
\end{proof}

\section{Proving Theorem \ref{thm:mainlist}}\label{sec:apps}
We now show that results for invariable generation apply to the families (i)--(v) introduced before Theorem \ref{thm:mainlist}.
\begin{proof}[Proof of Theorem~\ref{thm:mainlist}(1-3)]
We start with (1). The groups within categories (i)--(v) are in $\mathcal{MN}$, and hence we can apply Proposition \ref{prop:normal}.

For both (2) and (3), we show that \cite[Lem. 2.6(ii)]{KLS} applies (which says that if all maximal subgroups of $G$ are of bounded index, then $G$ is FIG). For (2), let $G$ be either the first Grigorchuk group or any torsion GGS-group. By~\cite[Thm.~1]{GW} and \cite[Thm. C]{FL}, the group $G$ has the subgroup induction property. By \cite[Thm.~B]{FL}, a finitely generated subgroup of $G$ has only maximal subgroups of finite index, and also only finitely many such maximal subgroups. We now show (3). Let $G$ be a torsion multi-EGS group. By~\cite[Cor.~5.7]{KT}, every finite-index  subgroup of $G$ has only finitely  many  maximal subgroups, all of finite index.
\end{proof}

Our final aim is to show Theorem~\ref{thm:mainlist}(4). 
First we recall some notation. For more information on groups acting on rooted trees, and in particular on branch groups, see~\cite{BarthGrigSunik}.

Let $T$ denote the binary rooted tree. We write $\psi$ for the natural embedding 
\begin{align*}
\psi \colon \st_{\Aut T}(1) &\longrightarrow 
\Aut  T \times 
\Aut T\\
g&\longmapsto (g_1,g_2)
\end{align*}
where $\St_{\Aut T}(1)$ is the set of automorphisms of $T$ which fix the first-level vertices. Here we also follow the natural ordering of subtrees rooted at the first-level vertices, so that the restriction of $g$ to the left subtree is $g_1$ and correspondingly $g_2$ acts on the right subtree.

For $m\ge 2$ let $f(x)=x^m+a_{m-1}x^{m-1}+\cdots +a_1x+a_0$ be an invertible polynomial over $\mathbb{F}_2$. The \v{S}uni\'{c} group $G_{2,f}$ is generated by the rooted automorphism $a$ corresponding to the
$2$-cycle $(1 \, 2 )$, and by the $m$ directed generators $b_1,\ldots,b_{m}\in \st_{\Aut T}(1)$ defined as follows:
\begin{align*}
    \psi(b_1)&=(1,b_2)\\
     \psi(b_2)&=(1,b_3)\\
     &\,\,\,\vdots\\
      \psi(b_{m-1})&=(1,b_{m})\\
       \psi(b_{m})&=(a,b_1^{-a_0}b_2^{-a_1}\cdots b_{m}^{-a_{m-1}}).
\end{align*}

We know from \cite[Thm.~1.1]{FG} that the non-torsion \v{S}uni\'{c} groups $G_{2,f}$ have maximal subgroups of infinite index, and from \cite[Cor.~2.13]{FG} these groups $G_{2,f}$ contain an element $b\in\langle b_1,\ldots,b_{m}\rangle$ such that $\psi(b)=(a,b)$.

For an odd prime $q$, note that $\{(ab)^q,b_1,\ldots,b_m\}$ is a generating set for a \v{S}uni\'{c} group~$G_{2,f}$. We recall the following.

\begin{lemma}\cite[Thm.~8.1]{FG} Let $G$ be a non-torsion \v{S}uni\'{c} group acting on the binary rooted tree. Then every maximal subgroup of infinite index is conjugate to $H(q):=\langle (ab)^q,b_1,\ldots,b_m\rangle$ for some odd prime $q$.
\end{lemma}

\begin{lemma}
    \label{lem:conjugate-of-ab}
    Let $G$ be a non-torsion \v{S}uni\'{c} group acting on the binary rooted tree. Then no conjugate of $ab$ is in $H(q)$ for any odd prime~$q$. 
\end{lemma}

\begin{proof}
Suppose for a contradiction that a conjugate of $ab$ is in $H(q)$ for some odd prime~$q$. Equivalently, suppose there is a maximal subgroup $M$ of $G$ of infinite index, with $ab\in M$. For a vertex $u$, recall that $\text{st}_M(u)$ denotes the subgroup consisting of elements in~$M$ that fix~$u$,   and for convenience, we write $M_u$ for the restriction of $\text{st}_M(u)$ to the subtree rooted at~$u$. Note that $ab\in M_u$ for every vertex $u$; compare \cite[Lem.~8.12]{FG}. Then by \cite[Lem.~8.15]{FG}, there is a vertex $v$ of the tree such that  $M_v=H(q)$. It follows that $ab\in H(q)$, which gives the desired  contradiction.
\end{proof}

\begin{proposition}\label{prop:Sunic2}
Let $G=\langle a,b_1,\ldots,b_m\rangle$, for some $m\ge 2$, be a non-torsion \v{S}uni\'{c} group acting on the binary rooted tree. Then $\{a,b_1,\ldots,b_m\}$ is not an IG-set, but $\{ab,b_1,\ldots,b_m\}$ is an IG-set. Hence $G$ is FIG.
\end{proposition}

\begin{proof}
For the first claim, note that $a^{(ba)^{\frac{q-1}{2}}}=(ab)^{q-1}a=(ab)^qb\in H(q)=\langle (ab)^q,b_1,\ldots,b_m\rangle$. Hence $\{a,b_1,\ldots,b_m\}$ is not an IG-set. Now, by Lemma~\ref{lem:conjugate-of-ab} no conjugate of $ab$ is in $H(q)$ for any odd prime~$q$. Therefore it follows that no conjugate of $ab$ is in any maximal subgroup of infinite index. It is straightforward that  $\{(ab)^{g_0},b_1^{g_1},\ldots,b_m^{g_m}\}$, for any $g_0,\ldots,g_m\in G,$ is not in any maximal subgroup of finite index, since these maximal subgroups are all normal and contain $G'$. Therefore $\langle (ab)^{g_0},b_1^{g_1},\ldots,b_m^{g_m}\rangle=G$.
\end{proof}

\section{Weakly generating sets and generating graphs} \label{sec:3/2}

Let $G$ be a group and $S\subseteq G$. Recall that $S$ \emph{weakly generates} $G$ if the image of $S$ in $G/G'$ is a generating set for $G/G'$. We prove Proposition~\ref{prop:weaklygen} via Lemma~\ref{lem:1966} and Lemma~\ref{lem:abgenonlyif} below. We thank the anonymous referee for the improvement to the following argument.

\begin{lemma} \label{lem:1966} Let $G\in\mathcal{MN}$ satisfy that every proper subgroup lies in a maximal subgroup and let $S\subseteq G$. If $S$ weakly generates $G$, then $S$ generates $G$.
\end{lemma}

\begin{proof}
By \cite[Thm.~A]{Myropolska}, we have that  $G'\le \Phi(G)$. 
Suppose that $S$ weakly generates $G$. Thus $\langle S\rangle G'=G$. Assume that $\langle S\rangle\ne G$. Then $\langle S\rangle$ is contained within some maximal subgroup $M$. But $G'\le \Phi(G)$ and so $G'\le M$, which contradicts that $\langle S\rangle G'=G$. Hence $\langle S\rangle=G$.
\end{proof}
For convenience, we also note an immediate consequence of the above.
\begin{corollary}\label{cor:abgen}
Let $G\in\mathcal{MN}$ be finitely generated. Then $d(G)=d(G/G')$.
\end{corollary}

\begin{lemma}\label{lem:abgenonlyif} Let $G$ satisfy that every proper subgroup lies in a maximal subgroup and suppose that $G\not\in \mathcal{MN}$. Then there is a set $T$ that weakly generates $G$ but but does not generate $G$.
\end{lemma}
\begin{proof} Take a generating set $S$ of $G$. Let $\pi_{\textrm{ab}}$ denote the abelianisation map from $G$ to $G/G'$. Therefore $X=\pi_{\textrm{ab}}(S)$ generates $G/G'$. Now, the choice $S'$ of a preimage of $X$ with respect to $\pi_{\textrm{ab}}$ corresponds to replacing each $s\in S$ with $sg'$ where $g'\in G'$. For any $g\in G$, we have that $s^{-1}g^{-1}sg$ is a possibility for $g'$. Hence if every choice of $S'$ generates $G$, then $S$ must be an IG-set. Proposition \ref{prop:normal} yields the result.
\end{proof}

Lemma \ref{lem:1966} can also be interpreted in terms of the \emph{generating graph}. Given a $2$-generated group $G$, we can define the generating graph $\Gamma(G)$ with vertex set $G\backslash \{1\}$ and an edge between $x$ and $y$ if and only if $\langle x, y\rangle=G$.  The generating graph has been well studied, especially for finite groups. Clearly $\frac32$-generation relates to `how connected' the generating graph is. It was proved in \cite{J} that for a finite group~$G$, either $\Gamma(G)$ has an isolated vertex or $\Gamma(G)$ is connected with diameter at most~2. At the time of writing, it is open whether there is a $2$-generated infinite group which is $\frac32$-generated but $\Gamma(G)$ has diameter greater than 2. It is also open whether a $\frac32$-generated group necessarily has $\Gamma(G)$ connected. For an arbitrary $2$-generated group~$G$ it is natural to consider the subgraph $\Delta(G)$ resulting from removing the isolated vertices from $\Gamma(G)$. 
The structure of $\Delta(G)$  even for finite groups is less straightforward since, for instance, there is no upper bound on its possible diameter; cf. \cite{CL}. 
Given $G$ in $\mathcal{MN}$ with $d(G)=2$, Lemma \ref{lem:1966} states that vertices in $\Delta(G)$ are connected if and only if their images in $\Delta(G/G')$ are connected. We note that \cite[Thm.~1.1]{AcciarriLucchini} describes all such generating graphs. Specifically, that $\Delta(\Z^2)$ is connected with infinite diameter and for any other abelian group $A$ with $d(A)=2$ we have that $\Delta(A)$ has diameter at most 2.   Therefore to study the diameter of the generating graphs of groups in $\mathcal{MN}$, it is sufficient to study the generating graphs of the possible abelianisations. Note that, by Corollary \ref{cor:abgen}, the abelianisation of our group~$G$ cannot be cyclic. We therefore have the following. 
\begin{theorem}\label{thm:diameters}
Let $G\in\mathcal{MN}$ with $d(G)=2$.
\begin{enumerate}
 \item If $G/G'\cong\Z^2$, then $\Delta(G)$ is connected but has infinite diameter.
 \item If $G/G'\cong C_2\times C_2$, then $\Delta(G)$ has diameter 1.
\item Otherwise $d(G/G')=2$ and $\Delta(G)$ has diameter 2. 
\end{enumerate}
\end{theorem}

We recall  that the \emph{total domination number} $\gamma_t(\Gamma)$ of a  graph $\Gamma$ is the least size of a set~$S$ of vertices of $\Gamma$  such that
every vertex of $\Gamma$  is adjacent to a vertex in $S$. Such a set~$S$ of least size is called a \emph{total dominating set} 
for  $\Gamma$. These concepts have been recently introduced for finite graphs, and  the total domination number of $\Delta(G)$ has been studied for $G$ a finite simple group or a finite 2-generated nilpotent group; see for example \cite{BH1} and \cite{nilpgengraph} respectively. As for infinite groups~$G$, not much is known about the total domination number and total dominating sets for~$\Delta(G)$, apart for the  Tarski monsters, which have a total dominating set of size 2, and   certain finite-index subgroups of the second Houghton group~\cite[Prop.~5.5]{Cox7}, which  have no finite total dominating sets. We make the following addition to the above list.

\begin{corollary}\label{cor:domset}
Let $G\in\mathcal{MN}$ with $d(G)=2$ and $G/G'\cong C_p\times C_p$ for some prime $p$. Then $\gamma_t(\Delta(G))=2$. Moreover, any lift of any generating set for $G/G'$ is a total dominating set for $\Delta(G)$.
\end{corollary}
\begin{remark} An example of a group $G$ satisfying the hypotheses in Corollary \ref{cor:domset} is any branch GGS-group.
\end{remark}

\end{document}